\documentclass[a4paper]{amsart}

\usepackage{latexsym, amssymb,amsmath, amsthm,amsfonts}

\newcommand{\kk}{\Bbbk}

\def\SL{\operatorname{SL}}

\def\SL2{\operatorname{SL}_{2}(K)}

\def\GL2{\operatorname{GL}_{2}(K)}

\def\INVSL2{$K[V]^{operatorname{SL}_{2}(K)}$}
\def\INVSO2{$K[V]^{operatorname{SO}_{2}(K)}$}
\def\INVGL2{$K[V]^{operatorname{GL}_{2}(K)}$}

\def\qed{\hfill $\Box$}

\def\GL{\operatorname{GL}}
\def\SL{\operatorname{SL}}
\def\id{\operatorname{id}}

\newtheorem{Lemma}{Lemma}
\newtheorem{Theorem}[Lemma]{Theorem}
\newtheorem{Proposition}[Lemma]{Proposition}
\newtheorem{Corollary}[Lemma]{Corollary}
\newtheorem{Prop}[Lemma]{Proposition}

\newtheorem*{Corollary of Conjecture}{Corollary of Conjecture}

\theoremstyle{definition}

\theoremstyle{remark}
  \newtheorem{rem}[Lemma]{Remark}
\newtheorem{eg}[Lemma]{Example}

\newtheoremstyle{Acknowledgments}
  {}
    {}
     {}
     {}
    {\bfseries}
    {}
     {.5em}
     {\thmname{#1}\thmnumber{ }\thmnote{ (#3)}}
\theoremstyle{Acknowledgments}

\title{Locally finite derivations and modular Coinvariants  }

\author{Jonathan Elmer}
\address{University of Aberdeen\\
King's College, Aberdeen\\
AB24 3UE}
\email{j.elmer@abdn.ac.uk}

\author{M\"{u}fit Sezer}
\address{Bilkent University, Department of Mathematics\\
Cankaya, Ankara \\06800 Turkey } \email{sezer@fen.bilkent.edu.tr}

\thanks{The second author is supported by a grant from T\"UBITAK:114F427 }
\date{\today}

\begin{document}
\maketitle

\begin{abstract}
We consider a finite dimensional $\kk
   G$-module $V$ of a $p$-group
$G$ over a field $\kk$ of characteristic $p$. We describe a
generating set for the corresponding Hilbert Ideal. In case $G$ is
cyclic this yields that the algebra $\kk[V]_G$ of coinvariants is
a free module over its subalgebra generated by $\kk G$-module
generators of $V^*$. This subalgebra is a quotient of a polynomial
ring by pure powers of its variables. The coinvariant ring was
known to have this property only when $G$ was cyclic of prime
order,  \cite{SezerCoinv}. In addition, we show that if $G$
  is the Klein 4-group and $V$ does not contain an indecomposable summand isomorphic to the
   regular module, then the Hilbert Ideal  is a complete
   intersection, extending a result of the second author and R. J. Shank \cite{SezerShank}.
\end{abstract}
\section{Introduction}

Let $\kk$ be a field of positive characteristic $p$ and $V$ a
finite-dimensional  $\kk$-vector space, and $G \leq \GL(V)$ a
finite group. Then the induced action  on $V^*$  extends to the
symmetric algebra $\kk[V]:=S(V^*)$ by the formula $\sigma
(f)=f\circ \sigma^{-1}$ for $\sigma \in G$ and $f\in \kk[V]$. The
ring of fixed points $\kk[V]^G$ is called the \emph{ring of
invariants}, and is the central object of study in invariant
theory. Another object which is often studied is the \emph{Hilbert
Ideal}, $\mathcal{H}$, which is defined to be the ideal of
$\kk[V]$ generated by invariants of positive degree, in other
words
\[\mathcal{H} = \kk[V]^G_+\kk[V].\]
In this article we study the quotient $\kk[V]_G:= \kk[V]/\mathcal{H}$ which is called the algebra of \emph{coinvariants}. An equivalent definition is $\kk[V]_G:= \kk[V] \otimes_{\kk[V]^G} \kk$, which shows that this object is, in a sense, dual to $\kk[V]^G$.

As $\kk[V]_G$ is a finite-dimensional $\kk G$-module, it is generally easier to handle than the ring of invariants. On the other hand, much information about $\kk[V]^G$ is encoded in $\kk[V]_G$. For example, Steinberg
 \cite {MR0167535} famously showed that $\dim(\mathbb{C}[V]_G) = |G|$ if and only if $(G,V)$ is a complex reflection group. Combined with the theorem of Chevalley \cite{MR0072877}, Shephard and Todd \cite{MR0059914}, this shows that
 $\dim(\mathbb{C}[V]_G) = |G|$ if and only if $\mathbb{C}[V]^G$ is a polynomial ring. Smith \cite{MR1972694} later showed that this holds independently of the ground field. Further, the polynomial property
 of $\mathbb{C}[V]^G$ is equivalent to the Poincar\'{e}
 duality property of $\mathbb{C}[V]_G$,
 by Kane \cite{MR1261561} and Steinberg \cite{MR0167535}.

Before we continue we fix some terminology. Let $x_0, \dots , x_n$
be a basis for  $V^*$. We will say $x_i$ is a \emph{terminal
variable} if the vector space spanned by the other variables is a
$\kk G$-submodule of $V^*$. Note that if $G$ is a $p$-group, then
 $V^G \neq 0$ and there is a choice of  a basis for $V$ that contains a
 fixed point. Then
 the dual element corresponding to the fixed point is a terminal
 variable in the basis consisting of dual elements of this basis.
 For any $f \in \kk[V]$ we define the norm $$N^G(f) = \prod_{h \in G \cdot f} h.$$
For every terminal variable $x_i$, we choose a polynomial $N(x_i)$ in $\kk[V]^G$ which, when viewed as a polynomial in
$x_i$ is monic of minimal positive  degree. While $N(x_i)$ is not
unique in general, its degree is well-defined. Since $N^G(x_i)$ is
monic of degree $[G:G_{x_i}]$ the degree of $N(x_i)$ is bounded
above by this number. By ``degree of $x_i$'' we understand degree of
$N(x_i)$ as a polynomial in $x_i$ and denote it by $\deg (x_i)$.
We will show that the degree of a terminal variable is always a
$p$-power.

 The algebras of modular coinvariants for cyclic groups
of order $p$ were studied by the second author \cite{SezerCoinv}
and previously by the second author and Shank \cite{MR2193198}.
Note that there is a choice of basis such that an indecomposable
representation of a $p$-group is afforded by an upper  triangular
matrix with 1's on the diagonal and the bottom variable  is a
terminal variable. In \cite{SezerCoinv} the following was proven.

\begin{Prop} Let $G$ be a cyclic group of order $p$ and $V$ a $\kk G$-module that contains $k+1$ non-trivial summands.   Choose a basis $x_0,x_1,\ldots ,x_n$ in which the variables $x_0,x_1,\ldots ,x_k$
are the bottom variables of the respective Jordan blocks,    and
let
 $A$ be the $\kk G$-subalgebra of $\kk[V]$ generated by $x_{k+1}, \ldots ,x_n$. Denote the image of $x_i$ in $\kk[V]_G$ by $X_i$. Then:

\begin{enumerate}
\item The Hilbert Ideal of $\kk[V]^G$ is generated by
$N^G(x_0),N^G(x_1),\ldots,N^G(x_k)$, and polynomials in $A$. \item
$\kk[V]_G$ has dimension divisible by $p^{k+1}$. \item $\kk[V]_G$
is free as a module over its subalgebra $\mathcal{T}$ generated by
$X_0,X_1,\ldots,X_k$. \item $\mathcal{T} \cong
\kk[t_0,\ldots,t_k]/(t_0^p,\ldots,t_k^p)$, where $t_0, \dots ,t_k$
are independent variables.
 \end{enumerate}
\end{Prop}

The goal of this article is to generalize the above, as far as
possible, to the case of all finite $p$-groups. In particular we
show in section two:

\begin{Theorem}
\label{central} Let $G$ be a finite $p$-group and $V$ a $\kk
G$-module that contains $k+1$ non-trivial summands. Choose a basis
$x_0,x_1,\ldots ,x_n$ in which the variables $x_0,x_1,\ldots ,x_k$
coming from each summand are terminal variables. Let $d_i$ denote
$\deg(x_i)$ for $0\le i\le k$. Retain the notation in the
proposition above, then:
\begin{enumerate}
\item There is  a choice for polynomials
$N(x_0),N(x_1),\ldots,N(x_k)$ such that the Hilbert Ideal of
$\kk[V]^G$ is generated by $N(x_0),N(x_1),\ldots,N(x_k)$, and
polynomials in $A$. \item $\kk[V]_G$ has dimension divisible by
$\prod_{i=0}^k d_i$.\
\end{enumerate}
Suppose in addition that,   one has $d_i=\deg(N^G(x_i))$ for $0\le
i\le k$. Then we have:

\begin{enumerate}
\item[(3)] $\kk[V]_G$ is free as a module over its subalgebra
$\mathcal{T}$ generated by $X_0,X_1,\ldots,X_k$. \item[(4)]
$\mathcal{T} \cong
\kk[t_0,\ldots,t_k]/(t_0^{d_0},\ldots,t_k^{d_k})$, where $t_0,
\dots ,t_k$ are independent variables.
\end{enumerate}
\end{Theorem}
In section three  we  describe a situation for a $p$-group, where
the complete intersection property of the Hilbert Ideal
corresponding to a module is inherited from the Hilbert Ideal of
the indecomposable summands of the module. The final section is
devoted to applications of our main results to cyclic $p$-groups
and the Klein 4-group. In turns out that for a  cyclic $p$-group
the bottom variables $x_i$ of Jordan blocks satisfy
$\deg(x_i)=\deg(N^G(x_i))$.
 Consequently, (3) and (4) above hold for a cyclic $p$-group.
  Additionally for the Klein 4-group we show that the Hilbert Ideal corresponding to a module is a complete intersection as long as the module  does not contain the regular module as a summand.
 This generalizes a result of the second author and Shank \cite{SezerShank}, where the complete intersection property was established  for indecomposable modules only.

This article was composed during a visit of the second author to the University of Aberdeen, funded by the Edinburgh Mathematical Society's Research Support Fund. We would like to thank the society for their support.

\section{Main Results}
\label{two}

Throughout this section, we let $G$ be a finite $p$-group, $\kk$ a
field of characteristic $p$ and $V$ a $\kk G$-module, which may be
decomposable. As trivial summands do not contribute to the
coinvariants, we assume no direct summand of $V$ is trivial. Let
$x_0,x_1,x_2,\ldots, x_n$ be a basis of $V^*$ and assume that
$x_0$ is  a terminal variable. Then    $x_1, x_2,\ldots, x_n$
generate a $G$-subalgebra which we denote by $A$. We can define a
non-linear action of $(\kk,+)$ on $\kk[V]$ as follows:

\begin{eqnarray}
t \cdot x_0 =& x_0+t  ;& \\
t \cdot x_i =& x_i & \text{for any $i>0$.}
\end{eqnarray}

The terminality of $x_0$ ensures this commutes with the action of $G$. It is well-known that any action of the additive group of an infinite field of prime characteristic is determined by a \emph{locally finite iterative higher derivation}. This is a family of $\kk$-linear maps $\Delta^i: \kk[V] \rightarrow \kk[V]$, $i \geq 0$ satisfying the following properties:

\begin{enumerate}
\item $\Delta^0= \id_{\kk[V]}$. \item For all $i>0$ and $a, b \in
\kk[V]$ one has $\Delta^i(ab) =
\sum_{j+k=i}\Delta^j(a)\Delta^k(b)$. \item For all $b \in \kk[V]$
there exists $i \geq 0$ such that $\Delta^i(b)=0$. \item For all
$i,j$ one has $\Delta^j \circ \Delta^i = \begin{pmatrix} i+j \\ j
\end{pmatrix}\Delta^{i+j}$.
\end{enumerate}

The equivalence of the group action and the l.f.i.h.d. is given by the formula
\begin{equation}\label{deltadef}
t \cdot b = \sum_{i \geq 0} t^i \Delta^i(b).
\end{equation}
See \cite{TanimotoLfihd, DufresneMaurischat} for more details on l.f.i.h.d.'s.

Let $f \in \kk[V]^G$ be homogeneous of degree $d$ in $x_0$. We
write
\[f= f_dx_0^d+f_{d-1}x_0^{d-1}+ \ldots+f_0,\] where $f_i \in A$.  We have
\[t \cdot f =f_d(x_0+t)^d+f_{d-1}(x_0+t)^{d-1}+ \ldots+f_0 =  \sum_{i \geq 0} t^i \Delta^i(f).\]
That is to say that $\Delta^i(f)$ is the coefficient of $t^i$ in the
above expression. As the action of $G$ commutes with the action of
 $\kk$, we see  that $\Delta^i(f) \in \kk[V]^G$ for all $i
\geq 0$.

\begin{rem}\label{rem}
\begin{enumerate}
\item Clearly $\Delta^1 = \frac{\partial}{\partial x_0}$. So the
previous paragraph  generalizes \cite[Lemma~1]{SezerCoinv}. \item
$\Delta^i$ is determined by the maps $\Delta^{p^j}$ for all $j$
such that $p^j<i$. One way to think of $\Delta^{p^j}$ is as
``Differentiation by $x_0^{p^j}$'': if the coefficient of $p^j$ in
the base $p$ expansion of $m$ is $a$,  then we have
\[\Delta^{p^j}(x_0^m) = \left\{ \begin{array}{lr} ax_0^{m-p^j} & a>0; \\0 & a=0. \end{array} \right.\]
A consequence of this is that for a homogeneous $f \in \kk[V]$,
$\Delta^j(f)$ is a non-zero constant if and only if its lead term
is $x_0^j$. \item In \cite{ElmerKohlsPPowers}  a $G$-equivariant
map is constructed from polynomials whose $x_0$-degree is at most
$ep^r$ $(0<e<p)$ to polynomials whose $x_0$-degree is at most
$p^r$. This map turns out to be a nonzero scalar multiple of
$\Delta^{(e-1)p^r}$.
\end{enumerate}
\end{rem}

We have the following statement generalizing
\cite[Lemma~2]{SezerCoinv}:
\begin{Lemma}\label{ident}
Let $f \in \kk[V]$ be a homogeneous polynomial of degree $d$ in
$x_0$. Write $f= f_dx_0^d+f_{d-1}x_0^{d-1}+ \ldots+f_0$, where
$f_i \in A$. Then we have
\[\sum_{i =0}^d (-1)^i x_0^i \Delta^i(f) = f_0.\]
\end{Lemma}

\begin{proof} Write $f = f(x_0,x_1,x_2,\ldots, x_n)$. For any $t \in \kk$ we have
\[t \cdot f = f(t \cdot x_0, t \cdot x_1,\ldots , t \cdot x_n) = f(x_0+t,x_1,x_2,\ldots,x_n). \]
As this holds for all $t$ it also holds when $t$ is replaced by $(-x_0)$, and hence by (\ref{deltadef})
$\sum_{i =0}^d (-1)^i x_0^i \Delta^i(f) = (-x_0)\cdot f = f(0,x_1,x_2,\ldots,x_n) = f_0$ as required.
\end{proof}

We also note that the degree of a terminal variable is a
$p$-power.
\begin{Lemma}
\label{power} For any terminal variable $x_0 \in V^*$, $\deg
(x_0)$   is a power of $p$.
\end{Lemma}
\begin{proof}
Let $d$ denote the degree of $x_0$ and suppose $f \in \kk[V]^G$ is
monic as a polynomial in $x_0$ of degree
$d=d_rp^r+d_{r-1}p^{r-1}+\cdots +d_0$ with $0\le d_i<p$ and
$d_r\neq 0$. If $d_j\neq 0$ for some $j<r$, then $\Delta^{p^j}(f)
\in \kk[V]^G$
  has degree $d-p^j>0$  as  a polynomial in $x_0$ and its leading coefficient is in $\kk$. Similarly,  if $d_j= 0$ for $j<r$ and $d_r>1$, then
$\Delta^{p^r}(f) \in \kk[V]^G$ has degree $d-p^r>0$ in $x_0$ and
its leading coefficient is in $\kk$. Both cases violate the
minimality of $d$.
\end{proof}

\begin{Lemma}
\label{hilbert} Let $d$ denote the degree of $x_0$. Then $\Delta^j
(\mathcal{H})\subseteq \mathcal{H}$ for $j<d$.
\end{Lemma}

\begin{proof}Let $f\in \kk[V]$.  From the second assertion of Remark \ref{rem} we get that $\Delta^j (f)$ contains a non-zero constant if and only if
the monomial $x_0^j$ appears in $f$.  Therefore, by the minimality
of $d$ we have $\Delta^j (\kk[V]^G_+)\subseteq \mathcal{H}$ for
$j<d$.  Now the result follows from property (2) of l.f.i.h.d.'s.
\end{proof}

From this point  on, we adopt the notation of the introduction.
This means that $x_0,x_1,x_2,\ldots ,x_k$ are terminal variables
coming from different summands,
 and $A = \kk[x_{k+1},x_{k+2},\ldots, x_n]$. For each $i=0,\ldots,k$ let $d_i = p^{r_i}$  be the degree of $x_i$. Since setting variables outside of a summand to zero sends invariants to invariants of the summand, we may also assume that
 $N(x_i)$ depends only on variables that come from the summand that contains $x_i$. We denote by $\Delta_i$ the l.f.i.h.d. associated to
 $x_i$. We use reverse lexicographic order with $x_i>x_j$ whenever
 $0\le i\le k$ and $k+1\le j\le n$.

\begin{Theorem}\label{main} $\mathcal{H}$ is generated by $N(x_0),\ldots,N(x_k)$ and polynomials in $A$. Moreover, the lead term ideal of $\mathcal{H}$ is generated
 by $x_0^{p^{r_0}},x_1^{p^{r_1}},\ldots, x_k^{p^{r_k}}$ and monomials in $A$.
\end{Theorem}

\begin{proof} Let $f \in \kk[V]^G$.   Since $N(x_0)$ is monic in $x_0$ we may perform polynomial division and
write $f = qN(x_0)+r$ where $r$ has $x_0$-degree $<p^{r_0}$, and
it is easily shown that $q,r \in \kk[V]^G$. Then dividing $r$ by
$N(x_1)$ yields another invariant remainder  $r'$  that has
$x_1$-degree $<p^{r_1}$. Since $x_0$-degree of $N(x_1)$ is zero,
it follows that $x_0$-degree of $r'$ is still $<p^{r_0}$. Thus,
 by repeating the process with each terminal variable, and replacing $f$ with the final remainder we assume that $x_i$-degree of $f$ is $<p^{r_i}$ for $0\le i\le k$.

Let $i$ be minimal such that $f$ has nonzero degree $d<p^{r_i}$ in
the terminal variable $x_i$. We apply Lemma \ref{ident} with
$\Delta=\Delta_i$ to see that

\[f=f_0-(\sum_{j =1}^d (-1)^j x_i^j \Delta_i^j(f)),\]
where $f_0$ is the ``constant term'' of $f$, i.e., $f_0 \in
\kk[x_{i+1},\ldots, x_n]$.
 So from the previous  lemma we get that  $f_0 \in
\mathcal{H}$ since  $d<p^{r_i}$. Moreover, since $\Delta_i$
decreases $x_i$-degrees and does not increase degrees in any other
variable, the $x_i$-degree of each $\Delta_i^j(f)$ in the
expression above is strictly less than $d$, and the $x_l$-degree
for every $i<l \leq k$ remains strictly less than $p^{r_l}$. Thus,
by induction on degree, $f$ can be expressed as a
$\kk[V]$-combination of elements of $\mathcal{H}$ whose degrees in
the terminal variables  $x_0, \dots ,x_i$ are all zero and degrees
in the remaining terminal variables $x_l$ for $i<l \leq k$ are
strictly less than $p^{r_l}$, respectively. Repeating the same
argument with the remaining terminal variables gives us
 that $f$ can be written as a $\kk[V]$-combination
of elements of $\mathcal{H} \cap A$ together with $N(x_1), \dots ,
N(x_k) $   as required. The first assertion of the theorem
follows.

Note that the leading monomial of $N(x_i)$ is $x_i^{p^{r_i}}$ for
$0\le i\le k$. So it remains to show that all other monomials in
the lead term ideal of  $\mathcal{H}$ lie in $A$.
 Recall that by Buchberger's algorithm a Gr\"{o}bner basis is obtained by reduction
    of $S$-polynomials of a generating set by polynomial division, see \cite[\S 1.7]{MR1287608}. By the first part, $\mathcal{H}$ has a generating set consisting of
   $N(x_i)$ for $0\le i\le k$ and polynomials in $A$.  But the $S$-polynomial of two polynomials in
   $A$ is also in $A$ and via polynomials in $A$ it also reduces to a polynomial in $A$. Finally, the $S$-polynomial of $N(x_i)$ and a
  polynomial in $A$ and the $S$-polynomial of a pair  $N(x_i)$  and  $N(x_j)$ with $0\le i\neq j \le k$  reduce to zero since their
  leading monomials
  are pairwise relatively prime.
\end{proof}
\begin{Corollary}
The vector space dimension of $\kk[V]_G$ is divisible by
$\prod_{0\le i\le k}d_i=p^{\sum_{i=0}^k r_i}$.
\end{Corollary}
\begin{proof}
The set of monomials that are not in the lead term ideal of
$\mathcal{H}$ form a vector space basis for $\kk[V]_G$. Let
$\Lambda$ denote this set of monomials. By the previous theorem a
monomial $M\in A$ lies in  $\Lambda$ if and only if
$Mx_0^{a_0}\cdots x_k^{a_k}$ lies in  $\Lambda$ for $0\le a_i<
p^{r_i}$ and $0\le i\le k$. It follows that the size of the set
$\Lambda$ is divisible by $p^{\sum_{i=0}^k r_i}$.
\end{proof}
The following generalizes the content of
\cite[Theorem~5]{SezerCoinv} partially for a $p$-group.

\begin{Theorem}
\label{constant} Let $x_i$ be a terminal variable of degree $d$,
and write $N(x_i) = x_i^d+\sum_{j=0}^{d-1} f_jx_i^j $, where
$x_i$-degree of $f_j$ is zero for $0\le j\le d-1$. Then $x_i^d+f_0
\in \mathcal{H}$.
\end{Theorem}

\begin{proof} Consider $\bar{N} = N(x_i)-x_i^{d}$. This is a polynomial of degree $e<d$ in $x_i$. By Lemma \ref{ident},
\[ \sum_{j =0}^e (-1)^j x_i^j \Delta_i^j(\bar{N}) = f_0 \]
since $f_0$ is the constant term of $\bar{N}$. Now recall that
$\Delta_i^j(x_i^{d})$ is the coefficient of $t^j$ in $(x_i+t)^{d}
= x_i^{d}+t^{d}$ (note
 that $d$ is a $p$-power by
Lemma \ref{power}). Thus, $\Delta^j(x_i^{d}) = 0$ for all $0<j<d$.
As $\Delta_i^j$ is a linear map for all $j$ it follows that
$\Delta_i^j(N(x_i)) = \Delta_i^j(\bar{N})$ for all $0<j<d$.
Therefore
\[\sum_{j =1}^e (-1)^j x_i^j \Delta^j({N(x_i)}) =f_0 -\bar{N}.\]
As $\Delta_i^j(N(x_i)) \in \mathcal{H}$ for all $j<d$ by Lemma
\ref{hilbert}, we get that $f_0-\bar{N} \in \mathcal{H}$.
Therefore $x_i^d+f_0 = N(x_i)-\bar{N}+f_0 \in \mathcal{H}$ as
required.
\end{proof}


\begin{Lemma}\label{freeness}
Suppose that for each $i=0,\ldots, k$ we have $x_i^{d_i} \in
\mathcal{H}$. Then $\kk[V]_G$ is free as a module over its
subalgebra $\mathcal{T}$ generated by $X_0,X_1,\ldots,X_k$, and
$\mathcal{T} \cong
\kk[t_0,\ldots,t_k]/(t_0^{d_0},\ldots,t_k^{d_k})$, where $t_0,
\dots ,t_k$ are independent variables.
\end{Lemma}

\begin{proof} The hypothesis on the $x_i$ is equivalent to $X_i^{d_i}=0$ in $\kk[V]_G$. Let
$t_0, \dots ,t_k$ be independent variables and consider the
natural surjective ring homomorphism from $\kk [t_0, \dots ,t_k]$
to $\kk [X_0, \dots ,X_k]$. Since $X_i^{d_i}=0$, the kernel of
this map contains $(t_0^{d_0}, \dots , t_k^{d_k})$. If this ideal
is not the all kernel, then $\mathcal{H}$ must contain a
polynomial in $x_0, \dots ,x_k$ such that no monomial in
 this polynomial is divisible by $x_i^{d_i}$ for
$0\le i\le k$. This is a contradiction with the description of the
lead term ideal in Theorem \ref{main}.

 Secondly, let $\Lambda$ denote the set of monomials  in $\kk[V]$ that are not in the
 lead term ideal of $\mathcal{H}$. Then the set of images of monomials in $\Lambda'=\Lambda \cup A$ generate $\kk[V]_G$ over $\mathcal{T}$. Further, they generate freely because
 $Mx_0^{a_0}\cdots x_k^{a_k}\in \Lambda$ for all $M\in \Lambda'$ and $0\le a_i<d_i$ and $0\le i\le k$ and the images of monomials in  $\Lambda$ form a vector space basis
 for $\kk[V]_G$.
\end{proof}

\begin{proof}[Proof of Theorem \ref{central}]
The first two assertions of the theorem are contained in Theorem
\ref{main} and its corollary. Next assume that
$d_i=\deg(N^G(x_i))$ for $0\le i\le k$. So we can take
$N(x_i)=N^G(x_i)$. Then from Theorem \ref{constant} it follows
that $x_i^{d_i}\in \mathcal{H}$ for $0\le i\le k$ since the
constant term of $N^G(x_i)$ (as a polynomial in $x_i$) is zero.
Now the third and the fourth assertions follow from Lemma
\ref{freeness}.
\end{proof}

\section{Complete intersection property of $\mathcal{H}$ }
 In this section we show that if the Hilbert Ideals of two
 modules are generated by fixed points and powers of terminal variables, then so is the Hilbert Ideal of the direct sum.
 As an incidental result we prove that the degree of a terminal variable does not change after taking direct sums.   We
continue with the notation and the convention of the previous
section. Let $V_1$ and $V_2$ be arbitrary $\kk G$-modules. We
choose a basis $x_{1,1}, \dots x_{n_1,1}, y_{1,1},\dots ,y_{m_1,1}
$ for $V_1^*$ and $x_{1,2}, \dots ,x_{n_2,2}, y_{1,2},\dots
,y_{m_2,2}$ for $V_2^*$ such that $x_{1,1}, \dots ,x_{n_1,1},
x_{1,2}, \dots ,x_{n_2,2}$ are fixed points.  Note that both
$\kk[V_1]$ and $\kk[V_2]$ are subrings of $\kk[V_1\oplus V_2]$ and
we identify
$$\kk[V_1\oplus V_2]=\kk [x_{1,1}, \dots ,x_{n_1,1},x_{1,2}, \dots
,x_{n_2,2}, y_{1,1},\dots ,y_{m_1,1}, y_{1,2},\dots ,y_{m_2,2}
].$$ Note that if $y_{i,j}$ is a terminal variable in $V_j^*$ for
some $1\le i\le m_j$, $1\le j\le 2$, then it is also a terminal
variable in $V_1^*\oplus V_2^*$.
\begin{Lemma}\label{sub}
 Assume the notation of the previous paragraph.  Let $y_{i,j}\in V_j^*$ be
a terminal variable. Then the degrees of $y_{i,j}$ in $V_j^*$ and
$V_1^*\oplus V_2^*$ are equal.
\end{Lemma}
\begin{proof}
Since $\kk[V_j]^G\subseteq \kk[V_1\oplus V_2]^G$, we have that the
degree of $y_{i,j}$ in $V_j^*$ is bigger than its degree in
$V_1^*\oplus V_2^*$. On the other hand, the restriction map
$\kk[V_1\oplus V_2]^G\rightarrow \kk[V_j]^G$ given $f\rightarrow
f_{|V_j}$ preserves any power of the form $y_{i,j}^d$. This gives
the reverse inequality.
\end{proof}
 We denote the Hilbert Ideals $\kk [V_1\oplus
V_2]^G_+\kk[V_1\oplus V_2]$,  $\kk[V_1]^G_+\kk[V_1]$ and
$\kk[V_2]^G_+\kk[V_2]$ with $\mathcal{H}$, $\mathcal{H}_1$ and
$\mathcal{H}_2$ respectively.
\begin{Theorem}\label{hilbertgens}
  Assume  that $\mathcal{H}_1$ and  $\mathcal{H}_2$ are generated by
the powers of the variables in $V_1^*$ and $V_2^*$, respectively
and that the variables    $ y_{1,1},\dots ,y_{m_1,1},
y_{1,2},\dots ,y_{m_2,2}$ are terminal variables. Then
$\mathcal{H}$ is  generated by the union of the generating sets
for $\mathcal{H}_1$ and  $\mathcal{H}_2$.
\end{Theorem}
\begin{proof}
Assume that $\mathcal{H}_1$ is generated by $x_{1,1}, \dots
,x_{n_1,1}, y_{1,1}^{d_{1,1}},\dots ,y_{m_1,1}^{d_{m,1}}$ and
$\mathcal{H}_2$ is generated by $x_{1,2}, \dots ,x_{n_2,2},
y_{1,2}^{d_{1,2}},\dots ,y_{m_2,2}^{d_{m,2}}$. We show that
$d_{i,j}$ is equal to the degree of the variable $y_{i,j}$ for
$1\le i\le m_j$ and $1\le j\le 2$. For simplicity we set $i=j=1$
and denote the degree of $y_{1,1}$ with $d$. Since $\mathcal{H}_1$
is generated by monomials, each monomial in a polynomial in
$\mathcal{H}_1$ is divisible by one of its monomial generators. So
we get $d_{1,1}\le d$. On the other hand, since
$y_{1,1}^{d_{1,1}}$ is a member of $\mathcal{H}_1$ there is  a
positive degree invariant with a monomial that divides
$y_{1,1}^{d_{1,1}}$. So by the minimality of $d$, we get $d\le
d_{1,1}$ as well. By Lemma \ref{sub}, $d_{i,j}$  is also equal to
the degree of $y_{i,j}$ in $\kk [V_1\oplus V_2]^G$. We claim that
the union of the generating sets for $\mathcal{H}_1$ and
$\mathcal{H}_2$ generate $\mathcal{H}$. Otherwise, there exists a
polynomial $f$ in $\mathcal{H}$ that contains a non-constant
monomial $\prod_{1\le i\le m_j, 1\le j\le 2}y_{i,j}^{e_{i,j}}$
with $0\le e_{i,j}< d_{i,j}$. Let $\Delta_{i,j}$ denote the
derivation with respect to the terminal variable $y_{i,j}$. Then
applying $\Delta_{i,j}^{e_{i,j}}$ successively to $f$ for $1\le
i\le m_j, 1\le j\le 2$ yields an invariant with a non-zero
constant. This is a contradiction by Lemma \ref{hilbert} since
$e_{i,j}< d_{i,j}$.
\end{proof}

We end this section with an example which shows that the degree of a terminal variable may be strictly less than the degree of its norm:
\begin{eg} Let $H = \langle \sigma, \tau \rangle$ be the Klein 4-group, $\kk$ a field of characteristic 2 and $m \geq 2$. Let
$\Omega^{-m}(\kk)$ be a vector space of dimension $m =2n+1$ over
$\kk$. Choose a basis $\{x_1,x_2,\ldots,
x_{m},y_{1},y_2,\ldots,y_{m+1}\}$ of $V^*$. One can define an
action of $H$ on $V$ in such a way that its action on $V^*$ is
given by $\sigma(y_j) = y_j+x_j, \sigma(x_j) = x_j, \tau(y_j) =
y_j+x_{j-1}, \tau(x_j)= x_j$ using the convention that $x_{0} =
x_{m} = 0$.

The variables $y_1,y_2,\ldots ,y_{m+1}$ are terminal. One can
readily check that
\[y_2^2+x_2y_2+x_1y_2+x_2y_1+x_1y_3+y_1x_3\] is invariant under $H$ (note the last term is zero if $m=2$), so $y_2$ has degree 2.
On the other hand, $y_2$ is not fixed by either $\sigma$ or
$\tau$, which means $N^H(y_2)$ has degree 4.  It is interesting to
note that
  $x_1y_2+x_2y_1+x_1y_3+y_1x_3 \in \mathcal{H}$, so we still have $y_2^2 \in \mathcal{H}$.
\end{eg}

\section{Cyclic $p$-groups and  the  Klein 4-group }
In this section we apply the results of the previous sections to
cyclic $p$-groups and the Klein 4-group. Let $G=Z_{p^{r}}$ denote
a cyclic group of order $p^r$. Fix a generator $\sigma$ of $G$.
There are  $p^r$ indecomposable $\kk G$-modules $V_1, \dots ,
V_{p^r}$ over $\kk$, and each indecomposable module $V_i$ is
afforded by $\sigma^{-1}$ acting via a Jordan block of dimension
$i$ with ones on the diagonal. For  an arbitrary $\kk G$-module
$V$, we write
\[
\quad\quad\quad V=\bigoplus_{i=0}^k V_{n_i} \quad\quad \text{
(with } 1\le n_{i}\le p^{r} \text { for all } i),
\]
where each $V_{n_i}$ is spanned as a vector space by $e_{1,i},
\dots , e_{n_i, i}$. Then the action of $\sigma^{-1}$ is given by
$\sigma^{-1}(e_{j,i})=e_{j,i}+e_{j+1,i}$ for $1\le j < n_i$ and
$\sigma^{-1} (e_{n_i,i})=e_{n_i,i}$. Note that the fixed point
space $V^G$ is $\kk$-linearly spanned by $e_{n_1,0}, \dots
,e_{n_k,k}$. The dual $V_{n_i}^*$ is isomorphic to $V_{n_i}$. Let
$x_{1,i}, \dots x_{n_i,i}$ denote the corresponding dual basis,
then we have
\[
\kk[V]=\kk[x_{j,i} \; \mid \,\, 1\le j\le n_i, \,\,\, 0\le i\le
k],
\]
and the action of $\sigma$ is given by
$\sigma(x_{j,i})=x_{j,i}+x_{j-1,i}$ for $1<j\le n_i$ and
$\sigma(x_{1,i}) =x_{1,i}$ for $0\le i\le k$.  Notice that the
variables $x_{n_i,i}$ for $0\le i\le k$ are terminal variables. We
follow the notation of Section \ref{two}
 and denote $x_{n_i,i}$ with $x_i$.
 We show that Theorem \ref{central} applies completely to $G$ by
 computing $\deg (x_i)$ explicitly for $0 \le i\le k$. For each $0\le i\le k $, let $a_i$ denote the largest integer such that
$n_i>p^{a_i-1}$.
 \begin{Lemma} We have $\deg(x_i)=p^{a_i}$. In particular, We may
take $N(x_i)=N^G(x_i)$.
 \end{Lemma}
\begin{proof}
  From \cite[Lemma 3]{KohlsSezerDegRed} we get that $\deg(x_i)$ is at least
  $p^{a_i}$. On the other hand since $p^{a_i}\ge n_i>p^{a_i-1}$, a Jordan block of size
  $n_i$ has order $p^{a_i}$. That is, this block affords a faithful module of the subgroup of $G$ of size $p^{a_i}$.
   It follows that the orbit of $x_i$ has $p^{a_i}$ elements and
   so that the orbit product $N^G(x_i)$ is a monic polynomial that is of degree
   $p^{a_i}$ in $x_i$.
\end{proof}
 Applying Theorem \ref{central}, we obtain the
following.
\begin{Proposition} Assume the notation of Theorem \ref{central} with specialization  $G=Z_{p^{r}}$.  We have an isomorphism
$$\kk [X_0, \dots ,X_k]\cong \kk [t_0, \dots ,t_k]/(t_0^{p^{a_0}}, \dots , t_k^{p^{a_k}}).$$
Moreover, $\kk[V]_G$ is free as a module over $\kk [X_0, \dots
,X_k]$. \qed
\end{Proposition}


Now  let $H$ denote the Klein 4-group and $p=2$. For each
indecomposable $\kk H$-module $V$ there exists a basis of $V^*$
with one of the terminal variables $x_i$ satisfying $\deg
(x_i)=[H:H_{x_i}]$, see \cite{SezerShank}. In this source it is
also proven  that, with the exception of the regular module, each
basis consists of fixed points and the terminal variables and the
Hilbert Ideal of every such module is generated by fixed points
and the powers of the terminal variables. So we have by Theorem
\ref{central} and Theorem \ref{hilbertgens}:

\begin{Proposition} Let $V$ be  a $\kk H$-module
   containing $k+1$ indecomposable summands.
   There is a
   basis $\{x_0,x_1,\ldots, x_n\}$ of $V^*$ in which $x_0,x_1,\ldots, x_k$ are terminal variables, each coming from one summand,
such that $\kk[V]_H$ is free as a module over its subalgebra
$\mathcal{T}$ generated by the images $X_0,X_1,\ldots,X_k$ of the
terminal variables. Moreover,  $\mathcal{T} \cong
\kk[t_0,\ldots,t_k]/(t_0^{a_0},\ldots,t_k^{a_k})$, where $t_0,
\dots ,t_k$ are independent variables, and for each $i$ we have
$a_i$ = 2 or 4.
\end{Proposition}

\begin{Proposition}
Let  $V$ be  a $\kk H$-module such that $V$ does not contain the
regular module $\kk H$ as a summand. Then there exists a  basis of
$V^*$ such that $\kk[V]^H_+\kk[V]$ is generated by powers of basis
elements. In particular, $\kk[V]^H_+\kk[V]$ is a complete
intersection.
\end{Proposition}

\bibliographystyle{siam}
\bibliography{MyBib}

\end{document}